\documentclass[letterpaper, 10 pt, conference]{ieeeconf} 

\IEEEoverridecommandlockouts                              
\newcommand{\w}{\omega}
\newcommand{\1}{\mathds 1  }

\newcommand{\pH}{port-Hamiltonian}
\newcommand{\G}{\mathcal{G} }
\newcommand{\V}{\mathcal{V} }

\newcommand{\E}{\mathcal{E} }
\newcommand{\comm}{c}
\usepackage[english]{babel}
\usepackage{amsmath}
\usepackage{amssymb}

\usepackage{graphicx}

\newtheorem{mythm}{Theorem}

\newtheorem{myrem}{Remark}

\usepackage[utf8]{inputenc}
\usepackage{dsfont}

\newtheorem{ass}{Assumption}

\title{\LARGE \bf
A Port-Hamiltonian Approach to Optimal Frequency Regulation in Power Grids
}

\author{Tjerk Stegink and Claudio De Persis and Arjan van der Schaft
\thanks{This work is supported by the NWO (Netherlands Organisation for Scientific Research) programme  \emph{Uncertainty Reduction in Smart Energy Systems (URSES)} under the auspices of the project ENBARK.}
\thanks{T.W. Stegink and C. De Persis are with the Engineering and Technology institute Groningen (ENTEG), Mathematics and Computer Science,
        University of Groningen, 9747 AG Groningen, The Netherlands.
        {\tt\small \{t.w.stegink, c.de.persis\}@rug.nl}}
\thanks{A.J. van der Schaft is with Johann Bernoulli Institute for Mathematics and Computer Science, University of Groningen, Nijenborgh 9, 9747 AG Groningen,
the Netherlands.
        {\tt\small a.j.van.der.schaft@rug.nl}}%
}

\begin{document}

\maketitle
\thispagestyle{empty}
\pagestyle{empty}

\begin{abstract}
This paper studies the problem of frequency regulation in power grids, while maximizing the social welfare. Two price-based controllers are proposed; the first one an internal-model-based controller and the second one based on a continuous gradient method for optimization.
Both controllers can be implemented in a fully distributed fashion, with freedom in choosing a controller communication network. As a result,  two real-time dynamic pricing models described by \pH\ systems are obtained. By coupling with the \pH\ description of the physical network we obtain a closed-loop \pH\ system, whose properties are exploited to prove asymptotic stability of the set of optimal points. Numerical results show the performance of both controllers in a simple case study.


\end{abstract}

\section{Introduction}

Stability of power networks is becoming an increasingly important topic in recent years. Especially with the growth of renewable energy sources there is an increasing fluctuation in the supply of power. As a result, it is more difficult for traditional energy sources to match the supply with the demand.  To alleviate some of these problems, we may introduce  a feedback mechanism that encourages the consumers to change their usage when it is difficult for the generators and the network to match demand. One approach is by using real-time dynamic pricing as a control method. 

The idea of using dynamic pricing to achieve optimal supply-demand matching is not new in the literature on power networks.  For a historical paper on dynamic pricing and market stability we refer to \cite{dynpricehis}. See e.g. also \cite{kiani_anna} and \cite{CDC2010_Stability} for more recent papers on real-time dynamic pricing, which mainly focus on the economic part of optimal supply-demand matching. However, the coupling between the solution of the optimization problem and the physical dynamics of the network should not be ignored as this could result in instability of the grid \cite{zhangpapaautomatica}. The coupling between the physics of the power network with the market dynamics has previously been studied in for example \cite{zhangpapaautomatica}, \cite{alv_meng_power_coupl_market},  and \cite{AGC_ACC2014}. 

In this paper, we propose a new approach for the modeling, analysis and control of smart grids based on using energy functions, both for the physical network as well as for the dynamic pricing algorithm. The underlying framework is based on the theory of port-Hamiltonian systems, which lends itself to the integration of dynamic pricing algorithms that allow to consider economical factors in the control of smart grids. The objective is to have producers and consumers to fairly share utilities and costs associated with the generation and consumption of power. The challenge of achieving this in an optimal manner is called the \emph{social welfare problem}. Simultaneously, the goal is to achieve zero frequency deviation w.r.t. to the nominal value (e.g. 50 Hz) in the power network. 

One of the approaches to solve an optimal frequency regulation problem is by using an internal-model-based controller as in \cite{swing-claudio,output_agreement}. 
We will continue along the same lines as in \cite{swing-claudio} which, among other things, treated optimal frequency regulation  in case of quadratic power production cost functions and constant unknown demand. The first main contribution of this paper is that we  extend the results of \cite{swing-claudio} where we will include a quadratic consumer utility function. The internal-model-based controller proposed in \cite{swing-claudio} is modified accordingly so that it steers the trajectories to the points of maximal social welfare while regulating the frequency. Moreover, this is all achieved within the \pH\ framework. In particular, we will show that the dynamics of the physical model of the power network as well as the real-time dynamic pricing model can be represented as \pH\ systems.



Another well-known controller design method for solving a(n) (social welfare) optimization problem is primal-dual gradient method based control. The literature on the gradient method has become quite extensive over the last decades, starting with the monograph \cite{arrow_gradmethod}.  Also in power grids this method is often applied to design distributed controllers, see for example \cite{zhangpapaautomatica,AGC_ACC2014} and  \cite{feijer-paganini}. 
Our contribution to the existing literature consists in showing that the real-time dynamic pricing model obtained when applying the gradient method can be represented as a \pH\ system, which demonstrates that the \pH\  framework can be extended from physical system modeling to markets dynamics as well. 

The outline of this paper is as follows. In Section \ref{sec:prel} we first state the preliminaries on the power network model and the social welfare problem. Next, we introduce an internal-model-based controller in Section \ref{sec:intmod}, and discuss its asymptotic stability. Thereafter in Section \ref{sec:gradmeth}, we propose a gradient method based controller in \pH\ form and we perform a similar stability analysis. Numerical results on both controllers will be discussed subsequently in Section \ref{sec:num}. Finally, we  state suggestions for future research.



\vspace{1cm}

\section{Preliminaries}\label{sec:prel}
\subsection{Power network model}
Consider a power grid consisting of $ n $ buses. The network is represented by a connected and undirected graph $ \G = (\V, \E) $, where the nodes, $ \V = \{1, . . . , n\} $, is the set of buses and the edges, $ \E \subset \V \times \V = \{1, . . . , m\} $, is the set of transmission lines connecting the buses.  The ends of edge $ k $ are arbitrary labeled with a ‘+’ and a ‘-’, so that the incidence matrix $D$ of the network is given by 
\begin{align*}
D_{ik}=\begin{cases}
+1 &\text{if $i$ is the postive end of $k$}\\
-1 &\text{if $i$ is the negative end of $k$}\\
0 & \text{otherwise.}
\end{cases}
\end{align*}
Each bus represents a control area and is assumed to
have controllable power generation and a price-controllable load. The dynamics at each bus are assumed to be given by \cite{swing-claudio}, \cite{powsysdynwiley}
\begin{equation}\label{eq:swingeq}
\begin{aligned}
\dot \delta_i&=\w_i^b-\w^n\\
M_i\dot \w_i&=u_{ig}-u_{id}-\sum_{j\in\mathcal N_i}V_iV_jB_{ij}\sin(\delta_i-\delta_j)\\&-A_i(\w_i^b-\w^n),
\end{aligned}
\end{equation}
which are commonly known as the \emph{swing equations}. We use here the following notations.
\begin{center}
\begin{tabular}{cl}
 $\delta_i$ & Voltage angle at bus $i$ \\ 
 $\w_i^b$ & Frequency at bus $i$ \\ 
 $\w^n$ & Nominal frequency \\ 
 $\w_i$ & Frequency deviation at bus $i$, i.e. $\w_i^b-\w^n$\\ 
 $V_i$ & Voltage at bus $i$\\ 
 $M_i$ & Moment of inertia at bus $i$ \\ 
 $A_i$ & Damping constant at bus $i$ \\ 
 $\mathcal N_i$ & Set of buses connected to bus $i$ \\
  $B_{ij}$ & Susceptance of the line between buses $i$ and $j$\\  
 $u_{di}$ & Power demand at bus $i$ \\ 
 $u_{gi}$ & Power generation at bus $i$ \\ 
\end{tabular}
\end{center}

\begin{ass}
By using the power network model \eqref{eq:swingeq} the following assumptions are made, which are standard in a broad range of literature on power network dynamics.
\begin{itemize}
\item Lines are lossless, i.e., the conductance is zero. This
assumption is generally valid for the case of high
voltage lines connecting different control areas.
\item  Nodal voltages $ V_i $ are constant.
\item  Reactive power flows are ignored.
\item  A balanced load condition is assumed, such that the
three phase network can be analyzed by a single phase.
\end{itemize}
\end{ass}
Define the voltage angle differences between the buses by $\eta=D^T\delta$.
Further define the angular momenta by $ p:=M\w$ where $\w=\w^b-\w^n$ the (aggregated)
frequency deviations and $M=\text{diag}(M_1,\ldots,M_n)$ are the moments of inertia. Let $\Gamma=\text{diag}(\gamma_1,\ldots,\gamma_k)$ and $\gamma_k=V_iV_jB_{ij}$ where $k$ corresponds to edge $(i,j)\in\E$. Finally, define the Hamiltonian $H_p(\eta,p)$ by 
\begin{align}\label{eq:hp}
H_p(\eta,p)=\frac12p^TM^{-1}p-\1^T\Gamma \cos \eta,
\end{align}
which consists of kinetic energy and a pendulum like potential energy.  The swing equations \eqref{eq:swingeq} are then represented by the  \pH\ system
\begin{equation}\label{eq:phswingeq}
\begin{aligned}
\begin{bmatrix}
\dot \eta\\
\dot p
\end{bmatrix}&=
\begin{bmatrix}
0&D^T\\
-D&-A
\end{bmatrix}
\nabla H_p(\eta,p)+
\begin{bmatrix}
0&0\\I&-I
\end{bmatrix}u\\
y&=\begin{bmatrix}
0&I\\
0&-I
\end{bmatrix}\nabla H_p(\eta,p)=\begin{bmatrix}
M^{-1}p\\-M^{-1}p
\end{bmatrix}=\begin{bmatrix}
\w\\-\w
\end{bmatrix}
\end{aligned}
\end{equation}
where $u=(u_g,u_d)$. Note that the system \eqref{eq:phswingeq} satisfies the passivity property
\begin{align*}
\dot H_p&=p^TM^{-1}\dot p+(\Gamma \sin \eta)^T\dot \eta=p^TM^{-1}(-D\Gamma\sin \eta\\&-AM^{-1}p+u_g-u_d)+(\Gamma \sin \eta)^TD^TM^{-1}p\\&=-\w^TA\w+\w^T(u_g-u_d)\leq u^Ty.
\end{align*}
For an extensive study on the stability and equilibria of the swing equations based on the Hamiltonian function \eqref{eq:hp}, we refer to \cite{swing-claudio}.

\subsection{Social welfare problem}
We define the social welfare by $U(u_d)-C(u_g)$, which consists of a  utility  function $U(u_d)$ of the consumers $u_d$ and the total power generation cost $C(u_g)$ associated to the producers $u_g$. The objective is to maximize the social welfare under the constraint of zero frequency deviation. We assume that $C(u_g)$ is a strictly convex function and $U(u_d)$ is a strictly concave function so that we will obtain an  optimization problem which is convex. 

By analyzing the equilibria of \eqref{eq:swingeq}, it follows that a necessary condition for zero frequency deviation is $\1^T u_d=\1^T u_g$ \cite{swing-claudio}, i.e.,  the total supply must match the total demand. It can be noted that $(u_g,u_d)$ is a solution to the latter equation if and only if there exists a $v\in \mathbb R^{m_c}$ such that $D_\comm v-u_g+u_g=0$ where $D_\comm\in\mathbb R^{n\times m_c}$ is the incidence matrix of some connected \emph{communication} graph with $m_c$ edges and $n$ nodes. This communication graph may be different from the physical network topology and will play a central role in the controllers proposed in Sections \ref{sec:intmod} and \ref{sec:gradmeth}. Because of the latter equivalence, we  consider the following convex minimization problem: 
\begin{equation}
\begin{aligned}\label{eq:minprob}
\min_{u_g,u_d,v} \ \  & R(u_g,u_d):=C(u_g)-U(u_d)\\
\text{s.t.}\ \  & D_\comm v-u_g+u_d=0.
\end{aligned}
\end{equation}
The corresponding Lagrangian is given by 
\begin{align*}
L = C(u_g)-U(u_d)+\lambda^T(D_\comm  v-u_g+u_d)
\end{align*}
with Lagrange multipliers $\lambda \in\mathbb R^n$. The resulting first-order optimality conditions ($\nabla L=0$)  are given by  
\begin{equation}\label{eq:KKTcond}
\begin{aligned}
	\nabla C(\bar u_g)-\bar \lambda         & =0  \\
	-\nabla U(\bar u_d)+\bar \lambda        & =0  \\
	D_\comm^T\bar \lambda                   & =0  \\
	D_\comm \bar  v-\bar u_g+\bar u_d       & =0.  
\end{aligned}
\end{equation}
Since the minimization problem is convex it follows that $(\bar u_g,\bar u_d,\bar v,\bar \lambda)$ is an optimal solution to \eqref{eq:minprob} if and only if it is a solution to \eqref{eq:KKTcond}, which is a standard result in the literature on convex optimization \cite{convexopt}.

\section{Internal-model-based controller}\label{sec:intmod}
In this section we extend the results of \cite{swing-claudio} in which we include a  utility function for the demand. 
We assume that the utility functions are quadratic and  given by $C(u_g)=\frac12u_g^TQ_gu_g+c^Tu_g, U(u_d)=-\frac12u_d^TQ_du_d+b^Tu_d$ where $Q_d,Q_g\in\mathbb R^{n\times n}$ are symmetric positive definite matrices and $c,b\in\mathbb R^n$. Consider the minimization problem \eqref{eq:minprob}. The first-order optimality conditions in this case amount to
\begin{equation}\label{feascond}
\begin{aligned}
	Q_g\bar u_g+c-\bar \lambda              & =0  \\
	Q_d\bar u_d-b+\bar \lambda              & =0  \\
	D^T_\comm \bar \lambda                  & =0  \\
	D_\comm \bar v-\bar u_g+\bar u_d        & =0.
\end{aligned}
\end{equation}
Note that $Q_gu_g+c$ are the (aggregated) marginal costs of the producers and likewise $-Q_du_d+b$ are the marginal utilities of the consumers. Observe from \eqref{feascond} that the prosumers  (combination of producers and consumers)  achieve maximal welfare if and only if their marginal costs and utilities respectively  are equal to the price $\lambda$, which is a standard result in economics. The optimal production and demand are therefore given by 
\begin{equation}\label{eq:uopt}
\begin{aligned}
\bar u_g&=Q_g^{-1}(\bar \lambda-c)\\
\bar u_d&=Q_d^{-1}(b-\bar \lambda).
\end{aligned}
\end{equation}
From the third equation of \eqref{feascond} it follows that the prices must be identical in each control area, i.e., $ \bar \lambda=\1 \lambda^*$, where the common price $\lambda^*$ is computed as
\begin{align}
\1^T\bar u_g&=\1^TQ_g^{-1}( \1\lambda^*-c)=\1^TQ_d^{-1}(b-\1  \lambda^*) =\1^T\bar u_d \nonumber\\
\Rightarrow \bar \lambda&=\1\lambda^*, \qquad \lambda^*=\frac{\1^T(Q_g^{-1}c+Q_d^{-1}b)}{\1^T(Q_g^{-1}+Q_d^{-1})\1} \label{eq:lambdastar}.
\end{align}
Based on the controller design proposed in \cite{swing-claudio}, we consider the following price-based controller dynamics in \pH\ form with inputs $u_\lambda$ and outputs $y_\lambda$: 
\begin{equation}\label{eq:lambdadyn}
\begin{aligned}
\dot \lambda&=-L_\comm\nabla H_c(\lambda)+\begin{bmatrix}
Q_g^{-1}&-Q_d^{-1}
\end{bmatrix}u_\lambda\\
y_\lambda
&=\begin{bmatrix}
Q_g^{-1}\\-Q_d^{-1}
\end{bmatrix}\nabla H_c(\lambda)+\begin{bmatrix}
-Q_g^{-1}c\\Q_d^{-1}b
\end{bmatrix}.
\end{aligned}
\end{equation}
The controller Hamiltonian is given by $H_c(\lambda)=\frac{1}{2}\lambda^T\lambda$ and $L_{\comm}=D_\comm D^T_\comm$ is the Laplacian matrix of the communication graph. We interconnect systems \eqref{eq:phswingeq} and \eqref{eq:lambdadyn} in a power-preserving way by $u_\lambda=-y, u=y_\lambda$. Then  we obtain the closed-loop \pH\ system 
\begin{align}
\begin{bmatrix}
\dot \eta\\
\dot p\\
\dot \lambda
\end{bmatrix}&=
\begin{bmatrix}
0&D^T&0\\
-D&-A&Q_g^{-1}+Q_d^{-1}\\
0&-Q_g^{-1}-Q_c^{-1}&-L_\comm
\end{bmatrix}
\nabla H(x)\nonumber\\&-\begin{bmatrix}
0\\Q_g^{-1}c+Q_d^{-1}b\\0
\end{bmatrix}\label{eq:clphsys1}
\end{align}
where $x=(\eta,p,\lambda)$ and the Hamiltonian is given by 
\begin{align*}
H(x)&=H_p(\eta,p)+H_c(\lambda)\\
&=\frac12p^TM^{-1}p-\1 ^T\Gamma\cos \eta+\frac{1}{2}\lambda^T\lambda.
\end{align*}
The equilibria  of \eqref{eq:clphsys1} are all  $(\bar \eta,\bar p,\bar \lambda)$ satisfying
\begin{equation}\label{eq:equi3}
\begin{aligned}
\bar p&=0\\
0&=-D\Gamma \sin \bar \eta+(Q_g^{-1}+Q_d^{-1})\bar \lambda-Q_g^{-1}c-Q_d^{-1}b\\
\bar \lambda&=\1\lambda^*, \qquad  \lambda^*=\frac{\1^T(Q_g^{-1}c+Q_d^{-1}b)}{\1^T(Q_g^{-1}+Q_d^{-1})\1}.
\end{aligned}
\end{equation}
We define $\Omega_1$ as the solution set of \eqref{eq:equi3}, i.e. 
\begin{align*}
\Omega_1&=\{ (\bar \eta,\bar p,\bar \lambda) \ | \  (\bar \eta,\bar p,\bar \lambda) \text{ is a solution to } \eqref{eq:equi3} \}.
\end{align*}
For proving local asymptotic stability of the closed-loop system \eqref{eq:clphsys1} an additional assumption is required.
\begin{ass} 
\label{ass:sec1}
There exists a $(\bar \eta,\bar p,\bar \lambda)\in \Omega_1$ such that $\bar \eta_k\in (- \pi/2, \pi/2)$ for all $k\in\E$.
\end{ass}

This  assumption is standard in studies on power grid stability and is also referred to as a security constraint \cite{swing-claudio-voltage}.

%
%
%
%
%

\subsection{Stability}
We will show that trajectories $(\eta,p,\lambda) $ satisfying \eqref{eq:clphsys1} and initialized sufficiently close to an equilibrium point of \eqref{eq:clphsys1} converge to the set $\Omega_1$. Moreover, we show that this set corresponds to the optimal points of the social welfare problem. 
\begin{mythm}\label{thm:1} For every $\bar x\in \Omega_1$ satisfying Assumption \ref{ass:sec1} there exists an open neighborhood $\mathcal O$ around $\bar x$ such that all trajectories $x$ satisfying \eqref{eq:clphsys1} with initial conditions in $\mathcal O$ converge to the set $\Omega_1$. Moreover, the power generations and demands converge to the optimal value given by \eqref{eq:uopt} and \eqref{eq:lambdastar}. 
\end{mythm}
\begin{proof}
Since  the system \eqref{eq:clphsys1} is not centered around the origin we introduce a shifted Hamiltonian $\bar H$ w.r.t. $\bar x\in \Omega_1$  \cite{swing-claudio},\cite{phsurvey}, which will act as a Lyapunov function: 
\begin{equation}\label{eq:shiftedHam}
\begin{aligned}
\bar H(x)&:=H(x)-(x-\bar x)^T\nabla H(\bar x)-H(\bar x)\\
&=\frac12p^TM^{-1}p-\1 ^T\Gamma\cos \eta+\frac{1}{2}\lambda^T\lambda\\
&-\begin{bmatrix}
\eta^T-\bar \eta^T&0&\lambda^T-\bar \lambda^T
\end{bmatrix}\begin{bmatrix}
\Gamma\sin \bar \eta\\0\\\bar \lambda
\end{bmatrix}\\
&-\1 ^T\Gamma\cos \bar \eta+\frac{1}{2}\bar \lambda^T\bar{\lambda}\\
&=\frac{1}{2}p^TM^{-1}p+\frac12(\lambda-\bar \lambda)^T(\lambda-\bar \lambda)\\&-\1 ^T\Gamma\cos \eta-(\eta-\bar \eta)^T\Gamma\sin \bar \eta+\1 ^T\Gamma\cos \bar \eta.
\end{aligned}
\end{equation}
Bearing in mind Assumption \ref{ass:sec1} and \cite{convexhamiltonianpersis}, the shifted Hamiltonian satisfies $\bar H(\bar x)=0$ and $\bar H(x)\geq0$ for all $x$ in an sufficiently small open neighborhood around $\bar x$. 
Moreover, the shifted Hamiltonian satisfies $\nabla \bar H( x)=\nabla H( x)-\nabla H(\bar x)$ so that \eqref{eq:clphsys1} can be rewritten as  
\begin{align*}
\dot x&=
\begin{bmatrix}
0&D^T&0\\
-D&-A&Q_g^{-1}+Q_d^{-1}\\
0&-Q_g^{-1}-Q_c^{-1}&-L_\comm
\end{bmatrix}
\nabla \bar H(x)\\&+\begin{bmatrix}
0\\
-D\Gamma\sin \bar \eta+(Q_g^{-1}+Q_d^{-1})\bar \lambda-Q_g^{-1}c-Q_d^{-1}b\\
-L_c\bar \lambda
\end{bmatrix}\\
&=\begin{bmatrix}
0&D^T&0\\
-D&-A&Q_g^{-1}+Q_d^{-1}\\
0&-Q_g^{-1}-Q_c^{-1}&-L_\comm
\end{bmatrix}
\nabla \bar H(x).
\end{align*}
Because of the \pH\ structure of the system it easily follows that the shifted Hamiltonian satisfies 
\begin{align}\label{eq:Ldiss}
\dot {\bar H}&=-\w^T A\w -(\lambda-\bar \lambda)^TL_\comm(\lambda-\bar \lambda)\leq 0,
\end{align}
where equality holds if and only if $\w=0$ and $\lambda=\bar \lambda+\1  \alpha$ for some scalar function $\alpha$. On the set $\dot {\bar H}=0$ we have
\begin{equation*}\label{eq:invset}
\begin{aligned}
\dot \eta&=0\\
\dot p&=-D\Gamma\sin \eta+D\Gamma\sin\bar \eta+(Q_g^{-1}+Q_d^{-1})\1  \alpha=0\\
\dot \lambda&=0
\end{aligned}
\end{equation*}
On the largest invariant set where  $\dot{\bar{H}}=0$ we must have that $\alpha\equiv0$, which  follows from premultiplication of the second equation by $\1 ^T$. Hence, by LaSalle's invariance principle $p\to0, \lambda\to \bar{\lambda}, \eta\to\hat \eta$ as $t\to \infty$, for some constant $\hat \eta$ satisfying 
\[
D\Gamma \sin\hat \eta=D\Gamma \sin \bar \eta=(Q_g^{-1}+Q_d^{-1})\bar \lambda-Q_g^{-1}c-Q_d^{-1}b.
\]
Hence $x\to \Omega_1$ as  $t\to\infty$. Moreover, since the interconnection between the controller and the swing equations \eqref{eq:phswingeq} is given by
\begin{align*}
 \begin{bmatrix}
 u_g\\u_d
 \end{bmatrix}=u=y_\lambda=\begin{bmatrix}
 Q_g^{-1}(\lambda-c)\\Q_d^{-1}(b-\lambda)
 \end{bmatrix}
\end{align*}
it follows that the power generations and demands converge to the optimal value given by \eqref{eq:uopt} and \eqref{eq:lambdastar} as $t\to\infty$.
\end{proof}


\section{Primal-dual gradient controller}\label{sec:gradmeth}
By applying the primal-dual gradient method \cite{zhangpapaautomatica}, \cite{AGC_ACC2014}, \cite{arrow_gradmethod} to the minimization problem \eqref{eq:minprob}, we obtain  the real-time dynamic pricing model
\begin{equation}\label{eq:graddistalvswing}
\begin{aligned}
\tau_{g}\dot u_g&=-\nabla C(u_g)+\lambda+w_g\\
\tau_d\dot u_d&=\nabla U(u_d)-\lambda+w_d\\
\tau_{ v}\dot  v&=-D_\comm^T\lambda\\
\tau_\lambda\dot \lambda&=D_\comm  v-u_g+u_d
\end{aligned}
\end{equation}
where we introduce additional inputs $w=(w_g,w_d)$ which are  to be specified later on. Here $\tau_E=\text{blockdiag}(\tau_g,\tau_d,\tau_v,\tau_\lambda)>0$ correspond to the timescales of the controller.  {Note that we have} constructed a distributed controller where $\lambda_i$ acts as a price in control area $i\in\V$ and $ v$ represents the information exchange of the differences of the prices $\lambda$ along the edges the communication graph.

Let us define the energy variables $x_E=(x_g,x_d,x_v,x_\lambda) = (\tau_gu_g,\tau_du_d,\tau_ v v,\tau_\lambda\lambda)=\tau_Ez_E$ and notice that in the sequel, we interchangeably write the system dynamics in terms of energy variables (denoted by $x$) and co-energy variables (denoted by $ z $) for ease of notation. 

An interesting fact is that the market dynamics \eqref{eq:graddistalvswing} admits a  \pH\  representation which is given by  
\begin{equation}\label{eq:dynpricsys}
\begin{aligned}
\dot x_E&=
\begin{bmatrix}
	0  & 0 & 0       & I          \\
	0  & 0 & 0       & -I         \\
	0  & 0 & 0       & -D_\comm^T \\
	-I & I & D_\comm & 0
\end{bmatrix}\nabla H_c(x_E)-\nabla R(z_E)\\
&+\begin{bmatrix}
I&0 \\
0& I\\
0&0\\
0&0
\end{bmatrix}w\\
y_E&=\begin{bmatrix}
I&0&0&0\\
0&I&0&0
\end{bmatrix}\nabla H_c(x_E)=\begin{bmatrix}
u_g\\u_d
\end{bmatrix},
\end{aligned}
\end{equation}
with the quadratic controller Hamiltonian 
\begin{align}\label{eq:hc}
H_c(x_E)=\frac{1}{2}x_E^T\tau_E^{-1}x_E.
\end{align}
Note that the latter system is indeed a \pH\ system since  $R$ is convex  and therefore satisfies the dissipativity property 
\begin{align*}
(z_1-z_2)^T(\nabla R(z_1)-\nabla R(z_2))\geq0, \ \forall z_1,z_2\in\mathbb R^{3n+m_c}.
\end{align*}
We obtain a power-preserving interconnection between \eqref{eq:phswingeq} and \eqref{eq:dynpricsys} by choosing $w=-y, u=y_E$. Define the extended vectors of(co-)energy variables by $x=(\eta,p,x_E), z=(\eta,\w,z_E)$ then the closed-loop \pH\ system takes the form
\begin{align}\label{eq:clphsys2}
\dot x&=\hspace{-1pt}
\begin{bmatrix}
0     &     D^T     &     0     &     0    &    0    &    0\\
-D     &     -A     &     I     &     -I    &    0    &    0\\
0     &     -I     &     0     &     0    &    0    &    I\\
0     &     I     &     0     &     0    &    0    &    -I\\
0     &     0     &     0     &     0    &    0    &    -D_\comm^T\\
0     &     0     &     -I     &     I    &    D_\comm    &    0
\end{bmatrix}\hspace{-2pt}\nabla H(x)-\nabla R(z)
\end{align}
with $H$  the sum of the energy function \eqref{eq:hp} corresponding to the physical model, and the controller Hamiltonian \eqref{eq:hc}.
We define the equilibrium set of \eqref{eq:clphsys2}, expressed in the co-energy variables, by 
\begin{align}\label{eq:omega2}
\Omega_2&=\{ \bar z \ | \  \bar z \text{ is an equilibrium of } \eqref{eq:clphsys2} \}.
\end{align}
It is noted that $\Omega_2$ is equal to the set of all $\bar z$ that satisfy the KKT optimality conditions \eqref{eq:KKTcond} and simultaneously satisfy the zero frequency constraints
\begin{align*}
-D\Gamma \sin \bar \eta+\bar u_g-\bar u_d=0, \quad \bar \w=0
\end{align*}
of the physical network \eqref{eq:phswingeq}. Hence, $\Omega_2$ corresponds to the desired equilibria. 


\begin{figure}
\begin{center}
\includegraphics[width=0.8\linewidth]{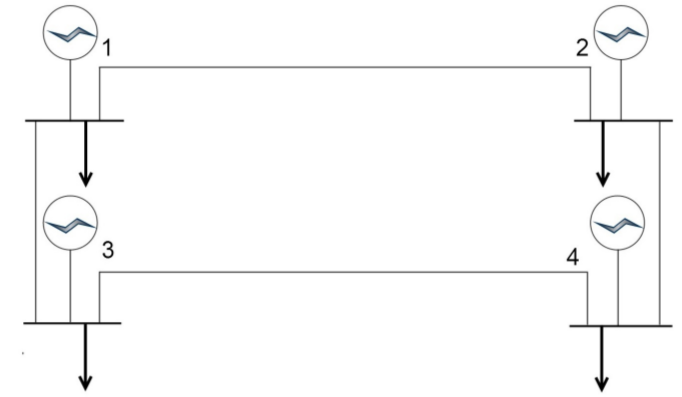}
\end{center}
\caption{A power grid consisting of 4 control areas \cite{swing-claudio}.}
\label{fig:4cont}
\end{figure}


\subsection{Stability}
\begin{mythm}  For every $\bar z\in \Omega_2$ there exists an open neighborhood $\mathcal O$ around $\bar z$ such that all trajectories $z$ satisfying \eqref{eq:clphsys2} with initial conditions in $\mathcal O$ converge to the set $\Omega_2$. 
\end{mythm}

\begin{proof} Let $\bar z\in\Omega_2$ and let this equilibrium be expressed in the energy variables by defining the vector $\bar x=(\bar \eta,0,\tau_g\bar u_g,\tau_d\bar u_d,\tau_v\bar v,\tau_\lambda\bar \lambda)$. 
Let the shifted Hamiltonian $\bar H$  around $\bar x$ be given by
\begin{align*}
&\bar H(x)=H(x)-(x-\bar x)^T\nabla H(\bar x)-H(\bar x)\\
&=\frac{1}{2}p^TM^{-1}p+\frac12(x_E-\bar x_E)^T\tau_E^{-1}(x_E-\bar x_E)\\&-\1 ^T\Gamma\cos \eta-(\eta-\bar \eta)^T\Gamma\sin \bar \eta+\1 ^T\Gamma\cos \bar \eta.
\end{align*}
As mentioned in the proof of Theorem \ref{thm:1}, it can be shown that  $\bar H(\bar x)=0$ and $\bar H(x)\geq0$ for all $x$ in a  sufficiently small open neighborhood around $\bar x$. After rewriting, the closed loop \pH\ system \eqref{eq:clphsys2} is equivalently described by 
\begin{align*}
\dot x&=
\begin{bmatrix}
0     &     D^T     &     0     &     0    &    0    &    0\\
-D     &     -A     &     I     &     -I    &    0    &    0\\
0     &     -I     &     0     &     0    &    0    &    I\\
0     &     I     &     0     &     0    &    0    &    -I\\
0     &     0     &     0     &     0    &    0    &    - D_\comm^T\\
0     &     0     &     -I     &     I    &    D_\comm    &    0
\end{bmatrix}\nabla \bar H(x)\\&-\nabla R(z)+\nabla R(\bar z).
\end{align*}
The shifted Hamiltonian $\bar H$ satisfies 
\begin{align*}
\dot{\bar H}&=-\w^T A\w-(z-\bar z)^T(\nabla R(z)-\nabla R(\bar z))\leq0
\end{align*}
where equality holds if and only if $\w=0, u_g=\bar u_g, u_d=\bar u_d$ since $R(z)$ is strictly convex in $u_g$ and $u_d$. On the largest invariant set $S$ where $\dot{\bar H}=0$ we have $\lambda=\bar \lambda$ and therefore $ v$ is constant. We conclude that $S\subset\Omega_2$ and by LaSalle's invariance principle it follows that $z\to S\subset \Omega_2$ as $t\to \infty$. 
\end{proof}



\subsection*{Comparison of both controllers}
When comparing both controllers it is noticed that the internal-model-based controller requires that the utility and cost functions are quadratic.
Since the matrices $Q_g,Q_d$ appear in the closed-loop interconnection structure \eqref{eq:clphsys1}, it would be challenging to generalize the internal-model-based controller to the case where  general strictly convex utility functions are considered. On the other hand, by applying the primal-dual gradient method to the optimization problem \eqref{eq:minprob} it is possible to construct a distributed controller that can deal with general convex utility functions.


\begin{myrem}\label{rem:lambda}
The controller variables $\lambda$ of both controllers are interpreted as the electricity prices, where we may have different prices in each of the control areas initially. Note that the controllers differ in the way they compensate for the price differences. On the one hand,  the $v$ dynamics of the gradient method based controller \emph{integrates} the differences between the prices $\lambda$. On the other hand, in the internal-model-based controller these differences are \emph{dissipated} through the Laplacian matrix $L_c$, by inferring from \eqref{eq:Ldiss} that as long as $\lambda$ is not in the range of $\1$, energy will be dissipated from the system. This has a stabilizing effect on the overall dynamics of the closed-loop system in case the internal-model-based controller is applied, see also Section \ref{sec:num}.
\end{myrem}

What both controllers have in common is that in the design of the distributed controllers there is freedom in choosing any communication graph, as long as the graph is connected.  Another remark is that, if we would assume that 
\begin{align*}
C(u_g)=\sum_{i=1}^{n}C_i(u_{gi}), \quad U(u_d)=\sum_{i=1}^{n}U_i(u_{di}),
\end{align*}
both distributed controllers in each control area require only information about their individual utility and cost functions, which is beneficial for privacy reasons.

\section{Numerical results}\label{sec:num}
We illustrate the performance of both proposed controllers, when applied to an academic test case, where we consider 4 control areas\footnote{This example is based on the 4 control area case study discussed in \cite{swing-claudio}.}, see Figure \ref{fig:4cont}. To compare both controllers we use identical quadratic utility and cost functions in the social welfare problem. The parameters  used for both cases are given by $Q_g=\text{diag}(1,2,3,4),M=\Gamma=Q_d=\frac12A=I,c=0, b=\text{col}(1,1.25,1.5,1.75), D_\comm=D, \tau_E=I$. 

We initialize at (optimal) steady operation, while at time $t=1$ we introduce a change in the utility function of the demand corresponding to area 4 by changing $b$ into $b=\text{col}(1,1.25,1.5,2)$, i.e., the consumption of electricity becomes more attractive in this area.  The simulations of the closed-loop systems are plotted in Figure \ref{fig:int} and \ref{fig:grad}. At steady state we  observe that the power production is higher in the control areas with lower costs functions and similar conclusions can be drawn for the power consumption. Hence at steady state, the social welfare is maximized. 

When the consumers utility function is changed, we observe that the consumption increases in control area 4, as we would expect. As a consequence, the common electricity price rises so that the power demands in the other control areas decrease. Simultaneously, the power production is increased to match the total supply and demand. It follows that the closed-loop dynamics converges again to the point where the social welfare is maximized.  It is noted that both controllers show comparable performance with the gradient method based controller showing a slightly more oscillatory behavior, which can be explained by Remark \ref{rem:lambda}.  





\begin{figure}
	\includegraphics[width=1.0\linewidth]{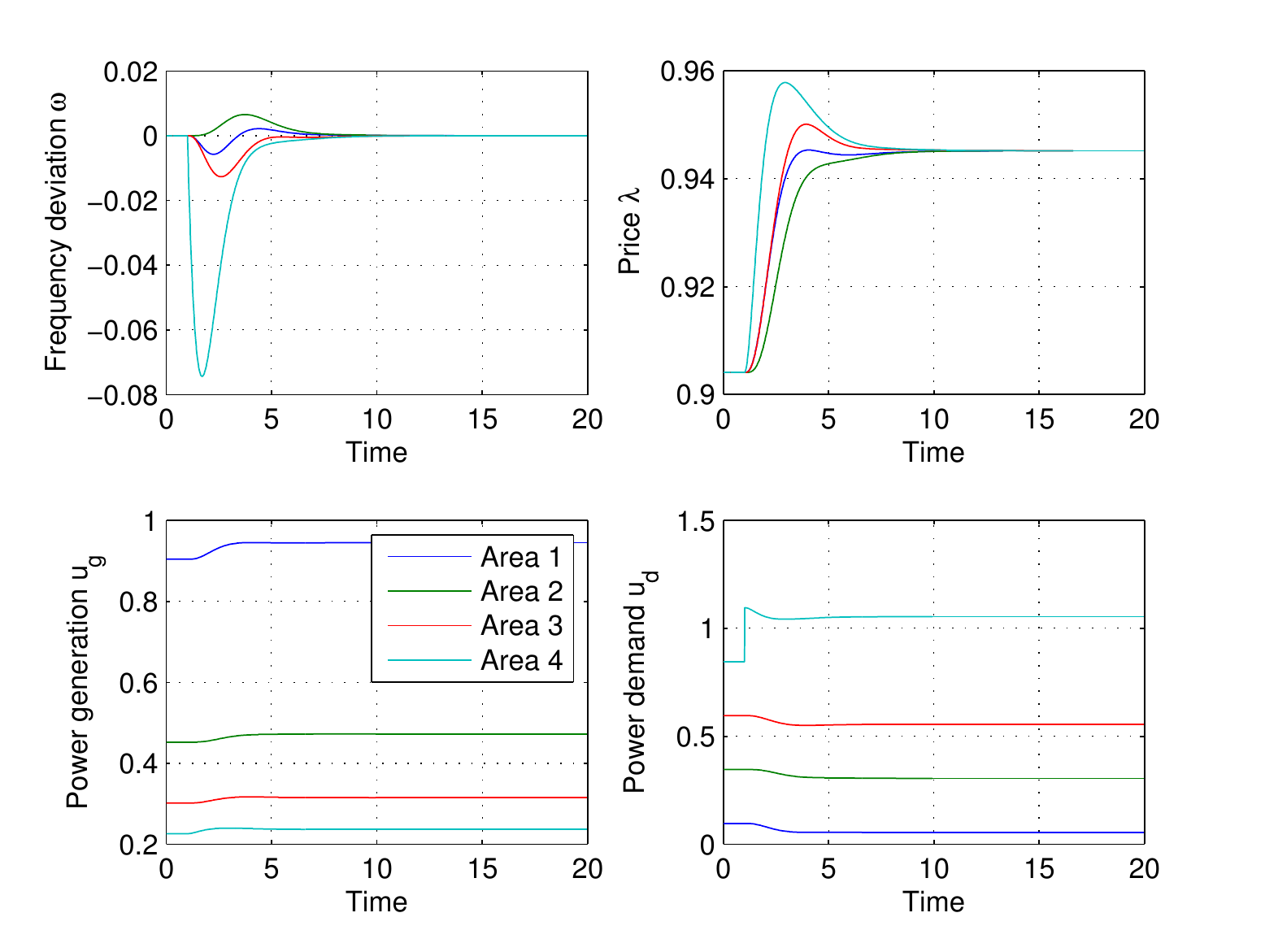}
\caption{Performance of the internal-model-based controller. At time $t=1$ the consumers utility function corresponding to control area 4 is changed, so that consumption of electricity becomes more attractive in this area.}
\label{fig:int}
\end{figure}

\begin{figure}
\includegraphics[width=\linewidth]{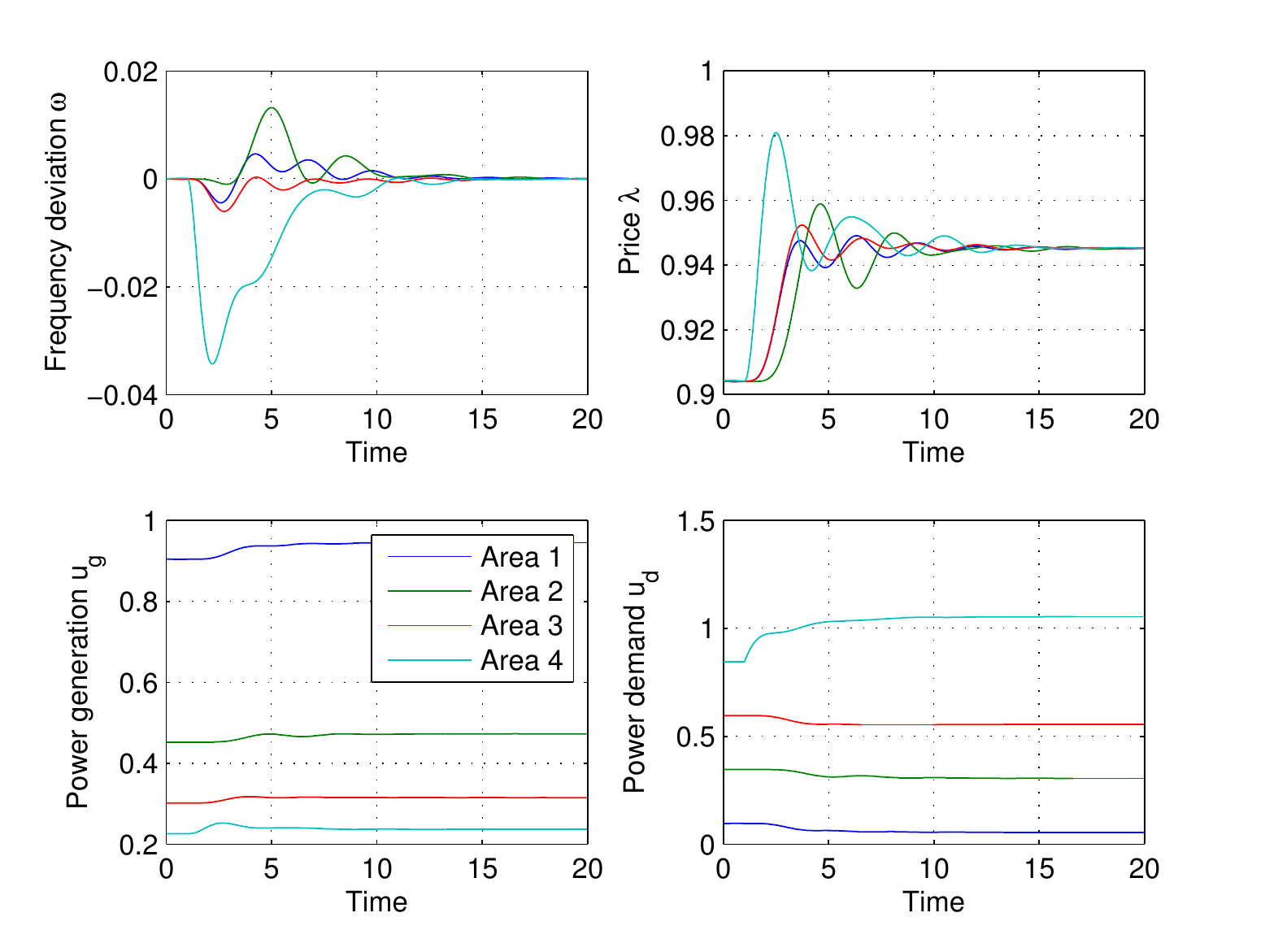}
\caption{Performance of the gradient method based controller. The same parameters and change in the utility function are considered as in the results shown in Figure \ref{fig:int}. Note that the gradient method based controller shows a slightly more oscillatory behavior, since it contains additional layers of integrators compared to the internal-model-based controller.}
\label{fig:grad}
\end{figure}

%


\section{Conclusions and future research}
In this paper we proposed a novel way of modeling, analysis and control of smart grids based on the \pH\ framework. We have proposed two different types of distributed real-time price-based controllers that achieve frequency regulation, while maximizing the social welfare. One controller is internal-model-based and the other is gradient method based. In the controller design there is freedom in choosing any connected communication graph. An important result is that  the market dynamics, obtained from applying the two proposed price-based controllers, can be represented in a \pH\ form. By coupling this with the \pH\ representation of the physical power network, 
%
a closed-loop \pH\ system is obtained whose properties are exploited to prove asymptotic stability to the set of optimal points.
By applying the real-time dynamic pricing models to the same (academic) test case, numerical results have shown the performance of both controllers in case of quadratic utility functions, and show convergence to the point where the  social welfare is maximized, even after a change in the consumers utility function.

\subsection*{Future research}
In this paper we only considered the total supply-demand matching constraints in the social welfare problem. A natural extension for future research is to  include additional inequality constraints, which for example correspond to congestion. 
Although the model for the power network used here is relatively simple, it provides a good starting point for considering more complex physical models of the power grid in the \pH\ framework. In these models one may for example include reactive power and voltage control. 
In addition to the proposed controller design methods, one would also like to develop controllers that can deal with uncertainties in the (demand) utility functions.


\bibliographystyle{IEEEtran}
\bibliography{./IEEEabrv,./db}

\end{document}